\documentclass[oneside, 10pt]{amsart}
\usepackage{latexsym, bbm, enumerate, amssymb, amsmath, wasysym}
\usepackage[round]{natbib}
\usepackage[dvips]{graphicx}
\usepackage{setspace}
\usepackage{paralist}
\usepackage{ifpdf}
\usepackage[capposition=top]{floatrow}

\newtheorem{Proposition}{Proposition}
\newtheorem{Lemma}{Lemma}

\newtheorem{Theorem}{Theorem}

\newcommand{\E}{{\mathbb{E}}}
\renewcommand{\P}{{\mathbb{P}}}

\newcommand{\FF}{{\mathcal{F}}}

\newcommand{\Var}{\mathrm{Var}}

\newcommand{\1}{\mathbbm{1}}

\newcommand{\subsectionnewline}{\mbox{}\medskip}

\DeclareMathOperator*{\argmin}{arg\,min}

\let\oldmarginpar\marginpar
\renewcommand\marginpar[1]{\-\oldmarginpar[\raggedleft\footnotesize #1]%
{\raggedright\footnotesize #1}}

\begin{document}


\title[Quickest Online Selection]%
{Quickest Online Selection of an Increasing Subsequence of Specified Size}
\author[A. Arlotto, E. Mossel, and Steele, J. M.]{Alessandro Arlotto, Elchanan Mossel,\\ and J. Michael Steele}

\thanks{A. Arlotto: The Fuqua School of Business, Duke University, 100 Fuqua Drive, Durham, NC, 27708.
Email address: \texttt{alessandro.arlotto@duke.edu}}

\thanks{E. Mossel:  Wharton School, Department of Statistics, Huntsman Hall
459, University of Pennsylvania, Philadelphia, PA 19104 and Department of Statistics, University of California, Berkeley, 401 Evans Hall,
Berkeley, CA 94720.
Email address: \texttt{mossel@wharton.upenn.edu}}

\thanks{J. M.
Steele:  Wharton School, Department of Statistics, Huntsman Hall
447, University of Pennsylvania, Philadelphia, PA 19104.
Email address: \texttt{steele@wharton.upenn.edu}}

\begin{abstract}

        Given a sequence of independent random variables with a common continuous distribution, we consider
        the online decision problem where one seeks to \emph{minimize the expected value of the time}
        that is needed to complete the selection of a monotone increasing subsequence of a prespecified length
        $n$.
        This problem is dual to some
        online decision problems that have been considered earlier,
        and this dual problem has some notable advantages. In particular,
        the recursions and equations of optimality lead
        with relative ease to asymptotic formulas for mean and variance of the minimal selection time.


\end{abstract}

\date{
first version: December 26, 2014;
this version: August 9, 2015. 
}
\maketitle

\section{Increasing Subsequences and Time-Focused Selection}

If $X_1, X_2, \ldots $ is a sequence of independent random variables with a common continuous distribution $F$, then
\begin{equation*}\label{eq:Ln-offline}
L_n= \max \{k: X_{i_1} < X_{i_2} < \cdots < X_{i_k}, \text{ where } 1 \leq i_1 < i_2 < \cdots < i_k\leq n\}
\end{equation*}
represents the length of the longest monotone increasing subsequence
in the sample $\{X_1, X_2, \ldots, X_n\}$.
This random variable has been subject to a remarkable sequence of investigations beginning with \citet{Ulam1961}
and culminating with
\citet{BaiDeiKur:JAMS1999} where it was found that $n^{-1/6}(L_n-2\sqrt{n})$ converges in distribution to the
Tracy-Widom law, a new universal law that was introduced just a few years earlier in \citet{TracyWidom1994}.
The review of \citet{AldDia:AMS1999} and the monograph of \citet{Rom:CUP2015} draw rich
connections between this increasing subsequence problem and topics
as diverse as card sorting, triangulation of Riemann surfaces, and --- especially --- the theory of integer partitions.

Here we consider a new kind of decision problem where one
seeks to select
\emph{as quickly as possible}
an increasing subsequence of a prespecified length $n$.
More precisely, at time $i$, when
the decision maker is first presented with $X_i$, a decision must be made either
to accept $X_i$ a member of the selected subsequence or else to reject $X_i$ forever.
The decision at time $i$ is assumed to be a deterministic function of
the observations $\{X_1, X_2, \ldots, X_i\}$, so the times $1\leq \tau_1 < \tau_2 < \cdots < \tau_n$
of affirmative selections give us a strictly increasing sequence stopping times that are adapted to the sequence of
$\sigma$-fields $\FF_i=\sigma\{X_1, X_2, \ldots, X_i\}$, $1 \leq i < \infty.$
The quantity of most interest is
\begin{equation}\label{def:beta}
\beta(n) := \min_\pi \E[\tau_n]
\end{equation}
where the minimum is over all sequences $\pi=(\tau_1, \tau_2,  \ldots, \tau_n)$  of stopping times such that
$$
1\leq \tau_1 < \tau_2 < \cdots < \tau_n  \quad \text{and} \quad
X_{\tau_1}< X_{\tau_2} < \cdots < X_{\tau_n}.
$$
Such a sequence $\pi$ will be called a \emph{selection policy}, and the set of all such selection policies with $\E [\tau_n] < \infty$
will be denoted by $\Pi(n)$.

It is useful to note that the value of $\beta(n)$ is not changed if we replace each $X_i$
with $F^{-1}(X_i)$, so we may as well assume from the beginning that the $X_i$'s are all uniformly distributed on $[0,1]$.
Our main results concern the behavior of $n \mapsto \beta(n)$
and the structure of the policy that attains the minimum \eqref{def:beta}.

\begin{Theorem}\label{thm:BetaAsym} The function
\begin{equation*}
n \mapsto \beta(n) = \min_{\pi\in \Pi(n)} \E[ \tau_n]
\end{equation*}
is convex, $\beta(1) = 1$,
and for all $n \geq 2$ one has the bounds
\begin{equation}\label{eq:betalimit}
\frac{1}{2} n^2 \leq \beta(n) \leq \frac{1}{2} n^2 + n \log n.
\end{equation}
\end{Theorem}

We can also add some precision to this result at no additional cost. We just need to
establish that there is always an optimal policy within
the subclass of \emph{threshold policies}
$\pi=(\tau_1, \tau_2,  \ldots, \tau_n) \in \Pi(n)$ that are determined by
real valued sequences $\{t_i \in [0,1]: 1 \leq i \leq n\}$ and the corresponding recursions
\begin{equation}\label{eq:defThresholdRecursion}
\tau_{k+1} = \min\{ i > \tau_{k}\,:  X_i \in [ X_{\tau_{k}}  , \, X_{\tau_{k}} + t_{n-k} (1-X_{\tau_{k}}) ] \, \}, \quad 0 \leq k < n.
\end{equation}
Here, the recursion begins with $\tau_{0}=0$ and $X_0=0$, and one can think of $t_{n-k}$ as the ``threshold parameter'' that
specifies the maximum fraction that one would be willing to spend from a ``residual budget'' $(1-X_{\tau_{k}})$  in order to accept a value
that arrives after the time $\tau_{k}$ when the $k$'th selection was made.

\begin{Theorem}\label{thm:Threshold Policy}
There is a unique threshold policy
$\pi^* =(\tau^*_1 , \tau^*_2, \ldots, \tau^*_n) \in \Pi(n)$ for which one has
\begin{equation}\label{eq:betalimit2}
\beta(n) = \min_{\pi\in \Pi(n)} \E[ \tau_n]=\E[ \tau^*_n],
\end{equation}
and for this optimal policy $\pi^*$ one has for all $\alpha>2$ that
\begin{equation}\label{eq:betavariance}
\Var [ \tau^*_n]=\frac{1}{3} n^3 + O(n^2 \log^\alpha n) \quad \text{as } n \rightarrow \infty.
\end{equation}
\end{Theorem}

In the next section, we prove the existence and uniqueness of an optimal threshold policy,
and in Section \ref{sec:mean} we complete the proof of Theorem \ref{thm:BetaAsym}
after deriving some recursions that permit the exact computation of the optimal threshold values.
Section \ref{sec:variance} deals with the asymptotics of the variance
and completes the proof of Theorem \ref{thm:Threshold Policy}.

In Sections \ref{se:BlockInquality} and \ref{sec:DualProblems},
we develop the relationship between
Theorems \ref{thm:BetaAsym} and \ref{thm:Threshold Policy} and the more traditional
\emph{size-focused} online selection problem which was
first studied in \citet{SamSte:AP1981} and then studied much more extensively by \citet{BruRob:AAP1991}, \citet{Gne:JAP1999},
\citet{BruDel:SPA2001,BruDel:SPA2004}, and \citet{ArlNguSte:SPA2015}.
On an intuitive level, the time-focused selection problem and the size-focused selection problems are dual to each other,
and it is curious to consider the extent to which rigorous relationships that can be developed between the two.
Finally, in Section \ref{sec:Conclusion} we underscore some open issues, and, in particular,
we note that there are several other selection problems where it may be beneficial to explore the possibility of
time versus size duality.

\section{Threshold Policies: Existence and Optimality}\label{sec:thresholdpolicies}

A beneficial feature of the time-focused monotone selection problem is that there
is a natural similarity between the problems
of size $n$ and size $n-1$.
This similarity entails a ``scaled regeneration" and leads to a useful recursion for $\beta(n)$.

\begin{Lemma}[Variational Beta Recursion]\label{lm:betaRecursion} For all $n=1,2,\ldots$ we have
\begin{equation}\label{eq:BetaRecursion}
\beta(n) = \inf_\tau \E \Big[\tau +  \frac{1}{1-X_\tau} \beta(n-1)\Big],
\end{equation}
where the minimum is over all stopping times $\tau$ and where we initialize the recursion
by setting $\beta(0)=0$.
\end{Lemma}
\begin{proof}

Suppose that the first value $X_\tau$ has been selected. When we condition on the value of $X_\tau$, we then
face a problem that amounts to the selection of $n-1$ values
from a \emph{thinned} sequence of uniformly distributed random variables on $[X_\tau, 1]$.
The thinned sequence may be obtained from the original sequence of uniform random variables on $[0,1]$
by eliminating those elements that fall in $[0,X_\tau]$.
The remaining thinned
observations are then separated by geometric time gaps that have mean $1/(1-X_\tau)$, so conditional on $X_\tau$, the
optimal expected time needed for the
remaining selections  is $\beta(n-1)/(1-X_\tau)$. Optimization over $\tau$ then yields \eqref{eq:BetaRecursion}.
\end{proof}

This quick proof of Lemma \ref{lm:betaRecursion} is essentially complete, but there is a variation on this argument that adds
further information and perspective.
If we first note that $\beta(1)=1$, then one can confirm \eqref{eq:BetaRecursion} for $n=1$
simply by taking
$\tau=1$. One can then argue by induction. In particular, we take $n\geq 2$ and consider a \emph{arbitrary} selection policy
$\pi=(\tau_1, \tau_2, \ldots, \tau_n)$.

If we set $\pi'= (\tau_2-\tau_1, \tau_3-\tau_1, \ldots, \tau_n -\tau_{1})$,
then one can view $\pi'$ as a selection policy for the sequence $(X_1',X_2', \ldots) =(X_{1+\tau_1}, X_{2+\tau_1}, \ldots)$
where one can only make
selections from those values that fall in the
interval $[X_{\tau_1}, 1]$.
As before, remaining thinned
observations are then separated by geometric time gaps that have mean $1/(1-X_\tau)$. Thus,
conditional on $X_\tau$, the \emph{optimal} expected time needed for the
remaining selections  is $\beta(n-1)/(1-X_\tau)$.
By the putative suboptimality of the selection times $\tau_1, \tau_2, \ldots, \tau_n$ one has
$$
\frac{\beta(n-1)}{ (1-X_{\tau_1})}
\leq \E[\tau_n-\tau_1 \,
| \, \tau_1, X_{\tau_1}]= -\tau_1+ \E[\tau_n\,| \, \tau_1, X_{\tau_1}].
$$
Moreover, we note that the sum
$$
\tau_1+ \frac{\beta(n-1)}{ (1-X_{\tau_1})}
$$
is the conditional expectation given $\tau_1$ of a strategy that first selects $X_{\tau_1}$ and then proceeds optimally with the selection
of $n-1$ further values.
The suboptimality of this strategy and the further suboptimality of the policy $\pi=(\tau_1, \tau_2, \ldots, \tau_n)$ give us
\begin{equation}\label{eq:betabound}
\beta(n) \leq \E[\tau_1+ \frac{\beta(n-1)}{ (1-X_{\tau_1})}]
\leq \E[\tau_n].
\end{equation}
The bounds of \eqref{eq:betabound} may now look obvious, but since they are valid for \emph{any} strategy one gains
at least a bit of new information. At a minimum,
one gets \eqref{eq:BetaRecursion} immediately
just by taking the infimum in \eqref{eq:betabound} over all $\pi$
in $\Pi(n)$.

Thus, in a sense, the bound \eqref{eq:betabound} generalizes the recursion \eqref{eq:BetaRecursion}
which has several uses. In particular, \eqref{eq:BetaRecursion}
helps one to show that there is a unique threshold policy that
achieves the minimal expectation $\beta(n)$.

\begin{Lemma}[Existence and Uniqueness of an Optimal Threshold Policy]\label{lm:Threshold}
For $ 1 \leq i \leq n$, there are constants
$0\leq t_i \leq 1$ such that the threshold policy
$\pi^* \in \Pi(n)$  defined by \eqref{eq:defThresholdRecursion} is the unique optimal policy.
That is, for $\pi^*=(\tau^*_1, \tau^*_2, \ldots, \tau^*_n)$ one has
\begin{equation*}\label{eq:OptimalThresholdStrategy}
\beta(n) = \min_{\pi \in \Pi(n)} \E[\tau_n]= \E[\tau^*_n],
\end{equation*}
and $\pi^{*}$ is the {\em only} policy in $\Pi(n)$ that achieves this minimum.
\end{Lemma}

\begin{proof}
The proof again proceeds by induction. The case $n=1$ is trivial since the only optimal policy is to take any element which presented.
This corresponds to the threshold policy with $t_1=1$ and $\beta(1) = 1$.

For the moment, we consider an arbitrary policy
$\pi=(\tau_1, \tau_2, \ldots, \tau_n) \in \Pi(n)$.
We have $1 \leq \E [\tau_1] < \infty$, and we introduce a parameter $t$ by setting $t=(\E[ \tau_1])^{-1}$.
Next, we define a new, threshold stopping time $\tau_1^*$ by setting
$$
\tau_1^* = \min\{i : X_i < t\},
$$
and we note that this construction gives us  $\E[\tau_1^*] = \E[\tau_1] =1/t$. For $s \in [0,t]$, we also have the trivial inequality
$$
\1(X_{\tau_1} <s) \leq \sum_{i=1}^{\tau_1} \1(X_i < s),
$$
so by Wald's equation we also have
\begin{equation}\label{eq:ByWald}
\P(X_{\tau_1} <s) \leq  \E[\sum_{i=1}^{\tau_1} \1(X_i < s)] = s \E[\tau_1] = s/t.
\end{equation}
The definition of $\tau^*_1$ implies that $X_{\tau^*_1}$ is uniformly distributed on $[0,t]$, so we
further have $\P(X_{\tau^*_1} <s)= \min\{ 1, s/t \}$.
Comparison with \eqref{eq:ByWald} then gives us the domination relation
\begin{equation}\label{eq:domination}
\P(X_{\tau_1} <s) \leq \P(X_{\tau^*_1} <s) \quad \text{for all }0 \leq s \leq 1.
\end{equation}
From \eqref{eq:domination} and the monotonicity of $x \mapsto (1-x)^{-1}$, we have by integration that
\begin{equation} \label{eq:domination2}
\E[\frac{\beta(n-1)}{1-X_{\tau_1^*}} )] \leq \E[\frac{\beta(n-1)}{1-X_{\tau_1}}],
\end{equation}
and one has a strict inequality in \eqref{eq:domination} and \eqref{eq:domination2} unless $\tau_1^* = \tau_1$
with probability one.

If we now add $\E[\tau_1^*] = \E[\tau_1]$ to the corresponding sides of
\eqref{eq:domination2} and take the infimum over all $\tau_1$, then the beta recursion \eqref{eq:BetaRecursion} gives us
$$
\E[\tau_1^*] +\E\Big[\frac{\beta(n-1)}{1-X_{\tau_1^*}} )\Big] \leq \inf_{\tau_1} \Big\{\E[\tau_1] +\E[\frac{\beta(n-1)}{1-X_{\tau_1}}] \Big\}
=\beta(n).
$$
In other words, the first selection of an optimal policy is given by uniquely by a threshold rule.

To see that all subsequent selections must be made by threshold rules, we just need to note that given the time
$\tau_1$ and value $X_{\tau_1}=x$ of the first selection, one is left with a selection problem of size $n-1$ from the
smaller set
$
\{X_i: i > \tau_1 \, \text{and} \, X_i>x \}.
$
The induction hypothesis applies to this problem of size $n-1$,
so we conclude that there is a unique threshold policy
$(\tau^*_2, \tau^*_3, \ldots, \tau^*_n)$ that is optimal for these selections. Taken as a whole, we have a unique threshold
policy $(\tau^*_1, \tau^*_2, \ldots, \tau^*_n) \in \Pi(n)$ for the problem of selecting an increasing subsequence of size $n$
in minimal time.
\end{proof}

Lemma \ref{lm:Threshold} completes the proof of the first assertion \eqref{eq:betalimit2} of Theorem \ref{thm:Threshold Policy}.
After we develop a little more information on the behavior of the mean, we will
return to the proof of the second assertion \eqref{eq:betavariance} of Theorem \ref{thm:Threshold Policy}.

\section{Lower and Upper Bounds for the Mean}\label{sec:mean}

The recursion \eqref{eq:BetaRecursion} for $\beta(n)$ is informative, but to determine its asymptotic behavior we need
more concrete and more structured recursions. The key relations are summarized in the next lemma.

\begin{Lemma}[Recursions for $\beta(n)$ and the Optimal Thresholds]\label{lem:ind}
For each $x \geq 1$ and $t \in (0,1)$ we let
\begin{equation}\label{eq:gGdef}
g(x,t) =  \frac{1}{t} + \frac{x}{t} \log\Big(\frac{1}{1-t}\Big), \quad
G(x) = \min_{0 < t < 1} g(x,t), \quad \text{and }
\end{equation}
\begin{equation}\label{eq:Hdef}
H(x) = \argmin_{0 < t < 1} g(x,t).
\end{equation}
We then have $\beta(1) = 1$, and we have the recursion
\begin{equation} \label{eq:ind}
\beta(n+1) = G(\beta(n))  \quad \text{for all } n \geq 1.
\end{equation}
Moreover, if the deterministic sequence $t_1, t_2, \ldots$ is defined by the recursion
\begin{equation}\label{tnDef}
t_1 = 1 \quad \text{and} \quad t_{n+1} = H(\beta(n)) \quad \text{for all } n\geq 1,
\end{equation}
then the minimum in the defining equation \eqref{def:beta} for $\beta(n)$
is uniquely achieved by the sequence of stopping times
given by the threshold recursion \eqref{eq:defThresholdRecursion}.
\end{Lemma}

\begin{proof}
An optimal  first selection time has the form $\tau_1 = \min\{i : X_i < t\}$, so we can rewrite the
recursion \eqref{eq:BetaRecursion} as
\begin{align}
\beta(n) &= \min_{0 < t < 1}\Big\{ \frac{1}{t} +  \E[\frac{\beta(n-1)}{1-X_{\tau_1}}] \Big\} =
\min_{0 < t < 1} \Big\{ \frac{1}{t}+ \frac{\beta(n-1)}{t} \int_{0}^t \frac{1}{1-s}\, ds \Big\} \notag\\
&=
\min_{0 < t < 1} g(t,\beta(n-1)) \equiv G(\beta(n-1)). \label{eq:gH}
\end{align}
The selection rule for the first element is given by
$
\tau_1= \min\{ i: X_i < t_n \}
$
so by \eqref{eq:gH} and the definitions of $g$ and $H$ we have $t_n = H(\beta(n-1))$.
\end{proof}

Lemma \ref{lem:ind} already gives us enough to prove
the first assertion of Theorem \ref{thm:BetaAsym}
which states that the map $n \mapsto \beta(n)$ is convex.

\begin{Lemma}\label{lm:BetaConvex}
The map $n\mapsto \Delta(n):= \beta(n+1) - \beta(n)$ is an increasing function.
\end{Lemma}

\begin{proof}
One can give a variational characterization of $\Delta$ that makes this evident.
First, by the defining relations  \eqref{eq:gGdef} and the recursion \eqref{eq:ind} we have
\begin{align*}
\beta(n+1) - \beta(n) & = G(\beta(n)) - \beta(n) \\
                      & = \min_{0< t < 1} \Big\{ \frac{1}{t} + \beta(n) \Big[ \frac{1}{t} \log \Big(\frac{1}{1-t}\Big) - 1 \Big] \Big\},
\end{align*}
so if we set
$$
\widehat g (x,t) = \frac{1}{t} + x \Big[ \frac{1}{t} \log \Big(\frac{1}{1-t}\Big) - 1 \Big],
$$
then we have
\begin{equation}\label{eq:BetaDifferenceRecursion}
\Delta(n)=\beta(n+1) - \beta(n) = \min_{0 < t < 1} \widehat g( \beta(n), t).
\end{equation}
Now, for $0\leq x \leq y$ and $t \in (0,1)$ we then have
$$
\widehat g(x, t) - \widehat g(y, t) = (x-y)\Big[ \frac{1}{t} \log \Big(\frac{1}{1-t} \Big) -1 \Big]
=(x-y)\sum_{k=2}^\infty \frac{1}{k} t^{k-1}\leq 0;
$$
so from the monotonicity $\beta(n) \leq \beta(n+1)$, we get
$$
\widehat g(\beta(n), t) \leq \widehat g(\beta(n+1),t) \quad \quad \text{for all } 0 < t < 1.
$$
When we minimize over $t\in (0,1)$, we see that \eqref{eq:BetaDifferenceRecursion} gives us $\Delta(n) \leq \Delta(n+1)$.
\end{proof}

We next show that the two definitions in \eqref{eq:gGdef} can be used to give an
\emph{a priori} lower bound on $G$. An induction argument using the recursion
\eqref{eq:ind} can then be used to obtain the lower half of \eqref{eq:betalimit}.

\begin{Lemma}[Lower Bounding $G$ Recursion]\label{lm:UpperBound}
For the function $x \mapsto G(x)$ defined by \eqref{eq:gGdef}, we have
\begin{equation}\label{eq:Gbound}
\frac{1}{2} (x+1)^2\leq G\Big( \frac{1}{2} x^2 \Big) \quad \quad \text{for all } x \geq 1.
\end{equation}
\end{Lemma}

\begin{proof}

To prove \eqref{eq:Gbound}, we first note that by \eqref{eq:gGdef} it suffices to show that
one has
\begin{equation}\label{eq:upper-bound-difference}
\delta(x,t) =(x+1)^2 t - 2 - x^2 \log\Big(\frac{1}{1-t}\Big)  \leq 0
\end{equation}
for all $x \geq 1$ and $t \in (0,1)$. For $x \geq 1$  the  map $t \mapsto \delta(x,t)$
is twice-continuous differentiable and concave in $t$. Hence
there is a unique value $t^* \in (0,1)$ such that $t^* = {\rm argmax}_{0<t<1} \delta(x,t)$,
and such that $t^*=t^*(x)$ satisfies the first order condition
$$
(x+1)^2 - (1 - t^*)^{-1} x^2  = 0.
$$
Solving this equation gives us
$$
t^* = \frac{2x + 1}{(x + 1)^2},
\quad \text{and} \quad
\delta(x,t^*) = - 1 + 2x - 2x^2 \log\Big( 1 + \frac{1}{x} \Big),
$$
so the familiar bound
$$
\frac{1}{x} - \frac{1}{2x^2} \leq \log\Big( 1 + \frac{1}{x} \Big) \quad \quad \text{for } x\geq 1,
$$
gives us
$$
\delta(x,t) \leq \delta(x,t^*) \leq   - 1 + 2x - 2x^2 \Big( \frac{1}{x} - \frac{1}{2x^2} \Big) = 0,
$$
and this is just what we needed to complete the proof of \eqref{eq:upper-bound-difference}.
\end{proof}

Now, to argue by induction, we consider the hypothesis that one has
\begin{equation}\label{eq:BetaNlowerbnd}
\frac{1}{2} n^2  \leq \beta(n).
\end{equation}
This holds for $n=1$ since $\beta(1)=1$, and, if it holds for some $n\geq 1$, then
by the monotonicity of $G$ we have $G(n^2/2) \leq G(\beta(n))$. Now, by \eqref{eq:Gbound}
and \eqref{eq:ind} we have
$$
\frac{1}{2} (n+1)^2 \leq G\Big(\frac{1}{2}n^2\Big) \leq G(\beta(n))=\beta(n+1),
$$
and this completes our induction step from \eqref{eq:BetaNlowerbnd}.

\subsection*{\sc An Alternative Optimization Argument\protect\footnote{This subsection is based on the kind suggestions of an anonymous referee.}}\subsectionnewline

From the definition of the threshold policy \eqref{eq:defThresholdRecursion}
we also know that for $0 \leq k < n$ the difference $ \tau_{k+1} - \tau_{k}$ of the selection times is a geometrically distributed
random variable with parameter $ \lambda_k = t_{n-k}(1-X_{\tau_k})$, so one has the representation
\begin{equation}\label{eq:beta-n-sum-geo}
\beta(n) = \sum_{k=0}^{n-1} \E\Big[ \frac{1}{\lambda_k} \Big].
\end{equation}
Moreover, we also have
$$
\E[X_{\tau_{k+1}} - X_{\tau_{k}} \, | \, X_{\tau_{k}}] = \frac{1}{2} \lambda_k
$$
so, when we take expectations and sum, we see that telescoping gives us
\begin{equation}\label{eq:LPproto2}
\frac{1}{2} \sum_{k=0}^{n-1} \E[ \lambda_k ]= \E[X_{\tau_n}] \leq 1,
\end{equation}
where the last inequality is a trivial consequence of the fact that  $X_{\tau_{n}} \in [0,1]$.

The relations \eqref{eq:beta-n-sum-geo} and \eqref{eq:LPproto2} now suggest a natural optimization problem.
Specifically, for $n \geq 2$ we let $\{l_k: 0 \leq k \leq n-1\}$ denote a sequence of random variables, and we consider the
following convex program
\begin{align}\label{eq:optimization-problem}
 & \text{minimize} \quad \sum_{k=0}^{n-1} \frac{1}{\E[l_k]} \\
 & \text{subject to} \quad \sum_{k=0}^{n-1} \E[l_k] \leq 2 \notag\\
 & \hphantom{\text{subject to}} \E[l_k] \geq 0,  \quad \quad  k = 0,2, \ldots, n-1.\notag
\end{align}
This is a convex program in the  $n$ real variables $\E[l_k]$,  $k=0,1, \ldots, n-1$, and,
by routine methods, one can confirm that its unique optimal solution is given by taking $\E[l_k] = 2/n$ for all $0 \leq k < n$.
The corresponding optimal value is obviously equal to $n^2/2$.

By the definition of the random variables $\{\lambda_k:  0\leq k < n\}$, we have the bounds $0\leq \lambda_k\leq 1$, and by \eqref{eq:LPproto2}
we see that their expectations satisfy the first constraint of the optimization problem \eqref{eq:optimization-problem},
so $\{\E[\lambda_k]:  0\leq k < n\}$ is a feasible solution of \eqref{eq:optimization-problem}.
Feasibility of the values $\{\E[\lambda_k]: 0 \leq k < n\}$ then implies that
\begin{equation}\label{eq:lambda-ell}
\frac{1}{2} n^2  \leq  \sum_{k=0}^{n-1} \frac{1}{\E\Big[  \lambda_k \Big]},
\end{equation}
and, by convexity of $x \mapsto 1/x$ on $[0,1]$, we obtain from Jensen's inequality that
\begin{equation}\label{eq:beta-n-sum-geo-JENSEN}
\sum_{k=0}^{n-1} \frac{1}{\E\Big[  \lambda_k \Big]}  \leq  \sum_{k=0}^{n-1} \E\Big[ \frac{1}{\lambda_k} \Big]=\beta(n).
\end{equation}
Taken together \eqref{eq:lambda-ell}
and \eqref{eq:beta-n-sum-geo-JENSEN} tell us that $n^2/2\leq \beta(n)$,
so we have a second proof of \eqref{eq:BetaNlowerbnd}, or
the lower bound \eqref{eq:betalimit} of Theorem \ref{thm:BetaAsym}.

Here we should note that \citet{Gne:JAP1999} used a similar convex programming argument to prove a tight
\emph{upper bound} in the
corresponding \emph{size-focused} selection problem.
Thus, we have here our first instance of the kind of duality that is discussed more fully in Section \ref{sec:DualProblems}.

\subsection*{\sc Completion of the Proof of Theorem \ref{thm:BetaAsym}}\subsectionnewline

We now complete the proof of Theorem \ref{thm:BetaAsym} by proving the upper half of \eqref{eq:betalimit}.
The argument again depends on an \emph{a priori} bound on $G$. The proof is brief but delicate.

\begin{Lemma}[Upper Bounding $G$ Recursion]\label{lm:SharpLowerBoundonBeta}
For the function $x \mapsto G(x)$ defined by \eqref{eq:gGdef} one has
$$
G( \frac{1}{2}x^2 + x \log(x) ) \leq \frac{1}{2} (x+1)^2 + (x + 1) \log(x + 1) \quad \quad \text{for all } x \geq 1.
$$
\end{Lemma}

\begin{proof}
If we set $f(x) : = x^2 / 2 + x \log (x)$, then we need to show that
$$
G(f(x)) \leq f(x+1).
$$
If we take $t' = 2/(x+2)$ then the defining relation \eqref{eq:gGdef} for $G$ tells
us that
\begin{equation}\label{eq:G-upperbound-1}
G(f(x)) \leq g(f(x),t') = \frac{x+2}{2} + \frac{x+2}{2} \log\Big(1 + \frac{2}{x}\Big) f(x).
\end{equation}
Next, for any $y \geq 0$ integration over $(0,y)$ of the inequality
$$
\frac{1}{u+1} \leq \frac{u^2 + 2u + 2}{2(u+1)^2} \quad \text{ implies the bound } \quad
\log ( 1 + y ) \leq  \frac {y(y+2)}{2 (y+1)}.
$$
If we now set $y = 2/x$ and substitute this last bound in \eqref{eq:G-upperbound-1}, we obtain
\begin{align*}
G(f(x)) & \leq \frac{x}{2} + 1 + \Big( 1 + \frac{1}{x} \Big) f(x) \\
        & = f(x+1) + \frac{1}{2} + (x+1)\{ \log(x) - \log(x+1) \} \leq f(x+1),
\end{align*}
just as needed to complete the proof of the lemma.
\end{proof}

One can now use Lemma \ref{lm:SharpLowerBoundonBeta} and induction to prove that for all
$n \geq 2$ one has
\begin{equation}\label{eq:BetaUpperBound}
\beta(n) \leq \frac{1}{2} n^2 + n \log(n).
\end{equation}
Since $\beta(2) = G(1) = \min_{0<t<1} g(1,t) < 3.15$ and
$2(1+\log(2))\approx 3.39$, one has \eqref{eq:BetaUpperBound} for $n=2$.
Now, for $n\geq 2$, the monotonicity of $G$ gives us that
$$
\beta(n) \leq \frac{1}{2} n^2 + n \log(n) \quad \text{ implies } \quad G(\beta(n)) \leq G\Big(\frac{1}{2} n^2 + n \log(n) \Big).
$$
Finally, by the
recursion \eqref{eq:ind} and Lemma \ref{lm:SharpLowerBoundonBeta} we have
\begin{align*}
\beta(n+1) &=G(\beta(n)) \leq G\Big(\frac{1}{2} n^2 + n \log(n) \Big) \\
&\leq  \frac{1}{2} (n+1)^2 + (n+1) \log(n+1).
\end{align*}
This completes the induction step and establishes \eqref{eq:BetaUpperBound} for all $n\geq 2$.
This also completes last part of the proof of Theorem \ref{thm:BetaAsym}.

\section{Asymptotics for the Variance}\label{sec:variance}

To complete the proof of Theorem \ref{thm:Threshold Policy}, we only need to prove that one has the
asymptotic formula \eqref{eq:betavariance} for  $\Var [ \tau^*_n].$ This will first require an understanding
of the size of the threshold $t_n$, and we can get this from our bounds on $\beta(n)$ once we have
an asymptotic formula for $H$. The next lemma gives us what we need.

\begin{Lemma} \label{lem:G}
For $x \mapsto G(x)$ and $x \mapsto H(x)$ defined by \eqref{eq:gGdef} and \eqref{eq:Hdef}, we have
for $x \to \infty$ that
\begin{equation} \label{eq:G_formula}
G(x) = ( x^{1/2} + 2^{-1/2} )^2 \{ 1 + O({1}/{x}) \} \quad \text{and}
\end{equation}
\begin{equation} \label{eq:H_formula}
H(x) = \sqrt{\frac{2}{x}} \{ 1 + O(x^{-1/2}) \}.
\end{equation}
\end{Lemma}

\begin{proof}
For any fixed $x \geq 1$ we have $g(t,x) \to \infty$ when $t \to 0$ or $t \to 1$, so the minimum of $g(t,x)$ is obtained at an
interior point $0 <t <1$. Computing the $t$-derivative $g_t(t,x)$ gives us
\[
g_t(t,x) = -\frac{1}{t^2} - \frac{x}{t^2} \log(\frac{1}{1-t}) + \frac{x}{t(1-t)},
\]
so the first-order condition $g_t(t,x) = 0$ implies that at the minimum we have the equality
\[
\frac{1}{t^2} = - \frac{x}{t^2} \log(\frac{1}{1-t}) + \frac{x}{t(1-t)}.
\]
Writing this more informatively as
\begin{equation} \label{eq:rootg}
\frac{1}{x} = \log(1-t) + \frac{t}{1-t} = \frac{t^2}{2} + \sum_{i=3}^{\infty} \frac{i-1}{i} t^i,
\end{equation}
we see the right-hand side is monotone in $t$, so there is a unique value ${t_*}={t_*}(x)$
that solves \eqref{eq:rootg} for $t$.
The last sum on the right-hand side of \eqref{eq:rootg} tells us that
\[
 \frac{1}{2} {t^2_*}  \leq  \frac{1}{x}  \quad \text{or, equivalently, } \quad {t_*} \leq \sqrt{\frac{2}{x}},
\]
and when we use these bounds in (\ref{eq:rootg}) we have
\[
\frac{{t^2_*}}{2} \leq \frac{1}{x} \leq \frac{{t^2_*}}{2} + \sum_{i=3}^{\infty} \Big(\frac{2}{x}\Big)^{i/2} \leq \frac{{t^2_*}}{2} + O(x^{-3/2}).
\]
Solving these inequalities for ${t_*}$, we then have by the definition \eqref{eq:Hdef} of $H$ that
\begin{equation}\label{HformulaAgain}
H(x) = {t_*} = \sqrt{\frac{2}{x}} \{ 1 + O(x^{-1/2}) \}.
\end{equation}
Finally, to confirm the approximation \eqref{eq:G_formula},
we substitute $H(x) = {t_*}$ into the definition \eqref{eq:gGdef} of $G$ and use the
asymptotic formula \eqref{HformulaAgain} for $H(x)$ to compute
\begin{align*}
G(x) = g(x,{t_*}) &= \frac{1}{{t_*}} + \frac{x}{{t_*}} \log\Big(\frac{1}{1-t_*}\Big)
= \frac{1}{{t_*}}
\Big(1 + x \sum_{i=1}^{\infty} \frac{t^i_*}{i} \Big) \\
&=
\frac{1}{{t_*}} \{ 1 + x {t_*} + \frac{x {t^2_*}}{2} + O(x {t_*}^3)\} =
\frac{1}{{t_*}}(x {t_*} + 1 + \frac{x {t^2_*}}{2})  + O(1) \\
&= x +  2\sqrt{\frac{x}{2}} + O(1) = \Big( x^{1/2} + 2^{-1/2} \Big)^2 \{ 1 + O(1/x) \},
\end{align*}
and this completes the proof of the lemma.
\end{proof}

The recursion \eqref{tnDef} tells us that $t_n=H(\beta(n))$ and the upper and lower bounds \eqref{eq:betalimit} of Theorem \ref{thm:BetaAsym}
tell us that $\beta(n)=n^2/2+O(n \log n)$, so by the asymptotic formula \eqref{eq:H_formula} for $H$ we have
\begin{equation}\label{eq:tnasymptotics}
t_n = \frac{2}{n} + O(n^{-2}{\log n}).
\end{equation}
To make good use of this formula we only need two more tools.
First, we need to note that the stopping time $\tau_n^*$ satisfies a natural distributional identity. This will lead in turn to
a recursion from which we can extract  the required asymptotic formula for $v(n) =\Var(\tau_n^*)$.

If
$t_n$ is the threshold value defined by the recursion \eqref{tnDef},
we let $\gamma(t_n)$ denote a geometric random variable of parameter $p=t_n$ and we let $U(t_n)$
denote a random variable with the uniform distribution on
the interval $[0,t_n]$. Now, if we take the random variables $\gamma(t_n)$, $U(t_n)$,
and $\tau_{n-1}$ to be independent, then we have the distributional identity,
\begin{equation}\label{eq:RVrepr}
\tau_n^* \stackrel{d}{=}  \gamma(t_n) + \frac{\tau_{n-1}^*}{(1-U(t_n))},
\end{equation}
and this leads
to a useful recursion for the variance of $\tau_n^*$. To set this up, we first put
$$R(t)=(1-U(t))^{-1}$$
where $U(t)$ is uniformly distributed on $[0,t]$, and we note
\begin{equation*}\label{eq:ERsq}
\E[R(t)]= - t^{-1} \log(1-t)=1+t/2+t^2/3+O(t^3);
\end{equation*}
moreover, since $\E[R^2(t)]= (1-t)^{-1} = 1+t +t^2+O(t^3)$, we also have
\begin{equation}\label{eq:VarR}
\Var[R(t)]= (1-t)^{-1} - t^{-2} \log^2(1-t)= \frac{t^2}{12} + O(t^3).
\end{equation}

\begin{Lemma}[Approximate Variance Recursion] \label{lem:var_recursion}
For the variance $v(n) := \Var(\tau_n^*)$ one has the approximate recursion
\begin{equation}\label{vn-recursion}
v(n) =  \Big\{1 + \frac{2}{n} + O(n^{-2} \log n ) \Big\} v(n-1) + \frac{1}{3}n^2 + O(n \log n).
\end{equation}
\end{Lemma}

\begin{proof}
By independence of the random variables on the right side of \eqref{eq:RVrepr}, we have
\begin{equation}\label{eq:varFirstStep}
v(n) = \Var(\gamma(t_n)) + \Var[R(t_n) {\tau_{n-1}^*}].
\end{equation}
From \eqref{eq:tnasymptotics} we have $t_n = 2/n + O(n^{-2}{\log n})$, so for the first summand we have
\begin{equation}\label{eq:vargamma}
\Var(\gamma(t_n))=\frac{1}{t_n^2} - \frac{1}{t_n} = \frac{1}{4} n^2 + O(n \log n).
\end{equation}
To estimate the second summand, we first use the complete variance formula and independence to get
\begin{align}
\Var(R(t_n) \tau_{n-1}^*) &= \E[R^2(t_n)] \E[(\tau_{n-1}^*)^2] - \E[R(t_n)]^2 (\E[\tau_{n-1}^*])^2  \notag\\
&=\E[R^2(t_n)] v(n-1) + \Var[R(t_n)] (\E[\tau_{n-1}^*])^2. \label{eq:2ndshoe}
\end{align}
Now from \eqref{eq:tnasymptotics} and \eqref{eq:VarR} we have
$$
\Var[R(t_n)] = \frac{1}{3} n^{-2} +O(n^{-3} \log n),
$$
and from \eqref{eq:betalimit} we have
$$
(\E[\tau_{n-1}^*])^2 = \Big\{\frac{1}{2}(n-1)^2+ O(n \log n) \Big\}^2= \frac{1}{4} n^4 +  O(n^3 \log n),
$$
so from \eqref{eq:varFirstStep}, \eqref{eq:vargamma} and \eqref{eq:2ndshoe} we get
\begin{align*}
v(n) 
= \{1 + \frac{2}{n} + O(n^{-2} \log n)\} v(n-1) + \frac{1}{3}n^2 + O(n \log n),
\end{align*}
and this completes the proof of \eqref{vn-recursion}.
\end{proof}

To conclude the proof of Theorem \ref{thm:Threshold Policy},
it only remains to show that the approximate recursion \eqref{vn-recursion}
implies the asymptotic formula
\begin{equation}\label{eq:betavariance2}
v(n)=\frac{1}{3} n^2(n+1) + O(n^2 \log^\alpha n) \quad \text{for } \alpha >2.
\end{equation}
If we define $r(n)$ by setting $v(n)= 3^{-1}n^2(n+1) + r(n)$, then substitution  of $v(n)$ into \eqref{vn-recursion}
gives us a recursion for $r(n)$,
$$
r(n) = \Big\{ 1 + \frac{2}{n} +O(n^{-2} \log n) \Big\} r(n-1) +O(n \log n).
$$
We then consider the normalized values $\hat{r}(n)=r(n)/(n^2 \log^\alpha(n) )$, and we note they satisfy the recursion
\begin{equation*}\label{eq:AdmittingDefeat}
\hat{r}(n) = \{ 1 + O(n^{-2}) \} \hat{r}(n-1) +O(n^{-1} \log^{\alpha-1} n).
\end{equation*}
This is a recursion of the form
$\hat{r}(n+1)= \rho_n \hat{r}(n) + \epsilon_n$,
and one finds by induction its solution has the representation
$$
\hat{r}(n)=\hat{r}(0)\rho_0\rho_1\cdots\rho_{n-1} + \sum_{k=0}^{n-1} \epsilon_k \rho_{k+1}\cdots\rho_{n-1}.
$$
Here, the product of the ``evolution factors" $\rho_n$ is convergent and the sum of the ``impulse terms" $\epsilon_n$ is
finite, so the sequence $\hat{r}(n)$ is bounded, and,  consequently, \eqref{vn-recursion}
gives us \eqref{eq:betavariance2}. This completes the proof of the last part of Theorem \ref{thm:Threshold Policy}.

\section{Suboptimal Policies and a Blocking Inequality}\label{se:BlockInquality}

Several inequalities for $\beta(n)$ can be obtained
through the construction of suboptimal policies.
The next lemma illustrates this method with an inequality that leads to an alternative proof
of \eqref{eq:BetaNlowerbnd}, the uniform lower bound for $\beta(n)$.

\begin{Lemma}[Blocking Inequality] For nonnegative integers $n$ and $m$ one has the inequality
\begin{equation}\label{eq:BetaMN}
\beta(mn) \leq \min \{ m^2\beta(n), n^2\beta(m)\}.
\end{equation}
\end{Lemma}
\begin{proof}
First, we fix $n$ and we consider a policy $\pi^*$ that achieves the minimal expectation $\beta(n)$. The idea is to use
$\pi^*$ to build a suboptimal $\pi'$ policy for the selection of an increasing subsequence of length $mn$. We take
$X_i$, $i=1,2,\ldots$ to be a sequence of independent random variables with the uniform distribution on $[0,1]$, and we
partition $[0,1]$ into the subintervals $I_1=[0,1/m)$, $I_2=[1/m,2/m)$, ..., $I_m=[(m-1)/m,1]$.
We then define $\pi'$ by three
rules:
\begin{enumerate}[(i)]
\item
    Beginning with $i=1$ we say $X_i$ is feasible value if $X_i \in I_1$. If $X_i$ is feasible, we accept $X_i$ if $X_i$ would be accepted by the policy $\pi^*$ applied to the sequence of feasible values after we rescale those values to be uniform in $[0,1]$.
    We continue this way until the time $\tau'_1$ when we have selected $n$ values.

\item
    Next, beginning with $i=\tau'_1$, we follow the previous rule except that now we say $X_i$ is feasible value if $X_i \in I_2$. We continue
    in this way until time $\tau'_1+\tau'_2$ when $n$ additional increasing values have been selected.

\item
    We repeat this process $m-2$ more times for the successive intervals $I_3$, $I_4$,..., $I_m$.
\end{enumerate}
At time $\tau'_1+\tau'_2+\cdots+\tau'_m$, the policy $\pi'$ will have selected $nm$ increasing values. For each
$1 \leq j \leq m$ we have $\E[\tau'_j]=m\beta(n)$, so by suboptimality of $\pi'$ we have
$$
\beta(mn) \leq \E[\tau'_1+\tau'_2+ \cdots + \tau'_m]=m^2 \beta (n).
$$
We can now interchange the roles of $m$ and $n$,  so the proof of the lemma is complete.
\end{proof}

The blocking inequality \eqref{eq:BetaMN} implies that even the crude asymptotic relation
$\beta(n)= n^2 / 2 + o(n^2)$  is strong enough to imply the uniform lower bound
$ n^2 / 2 \leq \beta(n)$. Specifically, one simply notes from \eqref{eq:BetaMN} and $\beta(n)= n^2 / 2 + o(n^2)$ that
$$
\frac{\beta(mn)}{(mn)^2} \leq \frac{\beta(n)}{n^2} \quad \text{and} \quad \lim_{m \rightarrow \infty} \frac{\beta(mn)}{(mn)^2}= \frac{1}{2} .
$$
This third derivation of the uniform bound $ n^2 / 2 \leq \beta(n)$
seems to have almost nothing in common with the
proof by induction that was used in the proof of Lemma \ref{lm:SharpLowerBoundonBeta}. Still, it does require the bootstrap bound
$\beta(n)= n^2 / 2 + o(n^2)$, and this does require at least some of the machinery of Lemma \ref{lem:ind}.

\section{Duality and the Size-Focused Selection Problem}\label{sec:DualProblems}

In the online \emph{size-focused} selection problem one considers
a set of policies $\Pi_s(n)$ that depend on the size $n$ of a sample $\{X_1,X_2, \ldots, X_n\}$,
and the goal is to make sequential selections in order to maximize the expected size of
the selected increasing subsequence. More precisely, a
policy $\pi_n \in \Pi_s(n)$ is determined by stopping times $\tau_i$, $i=1,2, ...$ such that
$1\leq \tau_1 < \tau_2 < \cdots < \tau_k \leq n$ and $X_{\tau_1} \leq  X_{\tau_2} \leq \cdots \leq X_{\tau_k}$.
The random variable of interest is
$$
L^o_n(\pi_n)=\max  \{ k:\;  X_{\tau_1} <  X_{\tau_2} < \cdots < X_{\tau_k}\text{ where }
                             1\leq \tau_1 < \tau_2 < \cdots < \tau_k \leq n \}, \notag
$$
and most previous analyses have focused on the asymptotic behavior of
\begin{equation*}\label{eq:OptimalThresholdStrategy2}
\ell (n) := \max_{\pi_n \in \Pi_s(n)} \E[L^o_n(\pi_n)].
\end{equation*}
For example, \citet{SamSte:AP1981} found that
$\ell (n) \sim \sqrt{2n}$,
but now a number of refinements of this are known. Our goal here is to
show how some of these refinements follow from the preceding theory.

\subsection*{\sc Uniform Upper Bound for $\ell(n)$ via Duality}\subsectionnewline

Perhaps the most elegant refinement of $\ell (n) \sim \sqrt{2n}$ is the following uniform upper bound
that follows independently from related analyses of \citet{BruRob:AAP1991} and \citet{Gne:JAP1999}.

\citet{BruRob:AAP1991} obtained the upper bound \eqref{eq:BRGagain} by exploiting the equivalence of the size-focused
monotone subsequence problem with a special knapsack problem that we will briefly revisit in Section \ref{sec:Conclusion}.
On the other hand, \citet{Gne:JAP1999} showed that  \eqref{eq:BRGagain}
can be proved by an optimization argument like the one used to prove \eqref{eq:optimization-problem}.
Here, we give a third argument with the curious feature that one gets a sharp bound that holds
for all $n\geq 1$ from an apparently looser relation that only holds approximately for large $n$.

\begin{Proposition}[Uniform Upper Bound]\label{prop:UnfUpper}
For all $n\geq 1$, one has
\begin{equation}\label{eq:BRGagain}
\ell (n)  \leq \sqrt{2n}.
\end{equation}
\end{Proposition}

As we noted above, this proposition is now well understood, but it still seems instructive to see how it can be derived from
$\beta(n) = ({1}/{2}) n^2 + O(n \log n)$.
The basic idea is that one exploits duality with a suboptimality argument like the one used in Section \ref{se:BlockInquality},
although, in this case a bit more work is required.

We fix $n$, and, for a much larger integer $k$,  we set
\begin{equation}\label{eq:Nk-AND-rk}
N_k = \lfloor (k-2k^{2/3})\ell(n) \rfloor \quad \text{and} \quad r_k= \lfloor k-k^{2/3} \rfloor.
\end{equation}
The idea of the proof is to give an algorithm that is guaranteed to select from $\{X_1, X_2, \ldots \}$
an increasing subsequence of length $N_k$. If $T_k$ is the number of the elements that the algorithm inspects
before returning the increasing subsequence, then by the definition of $\beta(\cdot)$ we have $\beta(N_k) \leq \E[T_k]$.
One then argues that \eqref{eq:BRGagain} follows from this relation.

We now consider $[0,1]$ and for $1\leq i \leq r_k,$ we consider the disjoint intervals $I_i=[(i-1)/k, i/k)$ and a
final ``reserve" interval $I^*=[r_k/k,1]$ that is added to complete the partition of $[0,1]$.
Next, we let $\nu(1)$ be the first integer such that $$\mathcal{S}_1 := \{ X_1, X_2, \ldots, X_{\nu(1)} \} \cap I_1$$
has cardinality $n$, and for each $i >1$ we define  $\nu(i)$ to be least integer greater $\nu(i-1)$ for which
the set $\mathcal{S}_i := \{ X_{\nu(i-1)+1}, X_{\nu(i-1)+2}, \ldots, X_{\nu(i)} \} \cap I_i$ has cardinality $n$.
By Wald's lemma and \eqref{eq:Nk-AND-rk} we have
\begin{equation}\label{eq:Wall1}
\E [ \nu(r_k) ] = n k r_k \quad \text{where } r_k= \lfloor k-k^{2/3} \rfloor.
\end{equation}

Now, for each $1 \leq i \leq n$, we run the \emph{optimal size-focused}
sequential selection algorithm on $\mathcal{S}_i$, and we let
$L(n,i)$ be the length of the subsequence that one obtains.
The random variables $L(n,i)$, $1 \leq i \leq r_k,$ are independent,
identically distributed, and each has mean equal to $\ell(n)$.
We then set
$$
\mathcal{L}(n,r_k) = L(n,1)+L(n,2) + \cdots + L(n,r_k),
$$
and we note that if  $\mathcal{L}(n,r_k)\geq N_k$, for $N_k$ as defined in \eqref{eq:Nk-AND-rk},
then we have extracted an increasing subsequence of length at least $N_k$;
in this case, we halt the procedure.

On the other hand if $\mathcal{L}(n,r_k)< N_k$, we need to send in the reserves.
Specifically, we recall that $I^*=[r_k/k,1]$ and
we consider the post-$\nu(r_k)$ reserve subsequence
$$
\mathcal{S}^* := \{ X_i: i > \nu(r_k) \, \,  \text{and} \, \, X_i \in I^*\}.
$$
We now rescale the elements of $\mathcal{S}^*$  to the unit interval, and we
run the optimal \emph{time-focused} algorithm on $\mathcal{S}^*$ until we get an increasing sequence of length $N_k$.
If we let $R(n,k)$ denote the number of observations from $\mathcal{S}^*$ that are examined in this case,
then we have $\E[R(n,k)]=\beta(N_k)$ by the definition of $\beta$. Finally, since $I^*$ has length at least $k^{-1/3}$, the expected number
of elements of $\{X_i: i > \nu(r_k) \}$ that need to be inspected before we have selected our increasing subsequence of
length $N_k$ is bounded above
by $k^{1/3}\beta(N_k)$.

The second phase of our procedure may seem wasteful, but one rarely needs to use the reserve subsequence.
In any event, our procedure does guarantee that we find an increasing subsequence of length $N_k$ in
a finite amount of time $T_k$. By \eqref{eq:Wall1} and the upper bound $k^{1/3}\beta(N_k)$ on the incremental cost
when one needs to use the reserve subsequence, we have
\begin{equation}\label{eq:BasicBetaBound}
\beta(N_k) \leq \E [T_k] \leq k n r_k + \{k n r_k+k^{1/3}\beta(N_k)\} \P(\mathcal{L}(n,r_k)< N_k),
\end{equation}
where, as noted earlier, the first inequality comes from the definition of $\beta$.

The summands of $\mathcal{L}(n,r_k)$ are uniformly bounded by $n$ and $\E [\mathcal{L}(n,r_k) ]= r_k \ell(n)$, so by  the
definition \eqref{eq:Nk-AND-rk} of $N_k$ and $r_k$ we see from Hoeffding's inequality that
\begin{align}\label{eq:ExponentialBound}
\P(\mathcal{L}(n,r_k)< N_k) &\leq \P\Big(\mathcal{L}(n,r_k) -\E[\mathcal{L}(n,r_k)] < -( k^{2/3} - 1 )\ell(n) \Big) \\
& 
\leq \exp\{ - A_n k^{1/3}\}, \notag
\end{align}
for constants $A_n$, $K_n$, and all $k \geq K_n$.
The exponential bound \eqref{eq:ExponentialBound} tells us that for each $n$ there is a constant $C_n$
such that the last summand in  \eqref{eq:BasicBetaBound} is bounded by $C_n$ for all $k$.
By the bounds \eqref{eq:betalimit} of Theorem \ref{thm:BetaAsym} we have $\beta(N_k)=(1/2) N_k^2 + O(N_k \log N_k)$,
and by \eqref{eq:Nk-AND-rk} we have
$$
N_k= (k-2k^{2/3})\ell(n) +O(1), \quad  r_k=k - k^{1/3} +O(1),
$$
so, in the end, our estimate \eqref{eq:BasicBetaBound} tell us
$$
\frac{1}{2} \ell^2(n) \{k^2 -2 k^{5/3} +4 k^{4/3}\} \leq k^2 n + o_n(k^2).
$$
When we divide by $k^2$ and let $k \rightarrow \infty$, we find $\ell(n) \leq \sqrt{2 n}$, just as we hoped.

\subsection*{\sc Lower Bounds for $\ell(n)$ and the Duality Gap}\subsectionnewline

One can use the time-focused tools to get a lower bound for $\ell(n)$, but in this case the slippage, or duality gap,
is substantial. To sketch the argument, we
first let $T_r$ denote the time required by the optimal time-focused selection policy to select $r$ values. We then
follow the $r$-target time-focused policy. Naturally, we stop if we have selected $r$ values, but if we have not selected $r$ values by time $n$,
then we quit, no matter how many values we have selected. This suboptimal strategy gives us the bound
$
r \P(T_r \leq n) \leq \ell(n),
$
and from this bound and Chebyshev's inequality,  we then have
\begin{equation}\label{eq:PostChebyshev}
r\{ 1- \Var(T_r)/(n-\E[T_r])^2 \} \leq \ell(n).
\end{equation}
If we then use the estimates \eqref{eq:betalimit} and \eqref{eq:betavariance} for  $\E [T_r]$ and $\Var[T_r]$ and optimize
over $r$, then \eqref{eq:PostChebyshev} gives us the lower bound $(2n)^{1/2} - O(n^{1/3})$.
However in this case the time-focused bounds and the duality argument leave a big gap.

Earlier, by different methods --- and for different reasons ---
\citet{RheTal:JAP1991} and \citet{Gne:JAP1999}  obtained the
lower bound $(2n)^{1/2} - O(n^{1/4}) \leq \ell(n)$.
Subsequently, \citet{BruDel:SPA2001} studied a continuous time interpretation
of the online increasing subsequence problem where the observations are presented to the decision maker at the arrival times
of a unit-rate Poisson process on the time interval $[0,t)$,
and, in this new formulation, they found the stunning lower bound $\sqrt{2t} - O(\log t)$.
Much later, \citet{ArlNguSte:SPA2015} showed by a de-Poissonization argument that the lower bound
of \citet{BruDel:SPA2001} can be used to obtain
$$
\sqrt{2n} - O(\log n) \leq \ell(n) \quad \quad \text{for all } n \geq 1
$$
under the traditional discrete-time model for sequential selection.
Duality estimates such as \eqref{eq:PostChebyshev}
seem unlikely to recapture this bound.

\section{Observations, Connections, and Problems}\label{sec:Conclusion}

A notable challenge that remains is that of understanding the asymptotic distribution of $\tau^*_n$,
the time at which one completes the selection of $n$ increasing values
by following the unique optimal policy $\pi^*$ that
minimizes the expected time $\E [ \tau^*_n ] = \beta(n)$.
By Theorems \ref{thm:BetaAsym} and \ref{thm:Threshold Policy} we know the behavior of the mean
$\E [ \tau^*_n ]$ and of the variance $\Var [\tau^*_n]$ for large $n$, but the asymptotic behavior of higher moments or
characteristic functions does not seem amenable to the methods used here.

Some second order asymptotic distribution theory should be available for $\tau^*_n$, but it is not likely to be easy.
Still, some encouragement can be drawn from the central limit theorem
for the size-focused selection problem
that was obtained by \citet{BruDel:SPA2004} under the Poisson model introduced in \citet{BruDel:SPA2001}.
Also, by quite different means, \citet{ArlNguSte:SPA2015} obtained the central limit theorem for the size-focused selection problem under the
traditional model discussed in Section \ref{sec:DualProblems}. Each of these developments depends on martingale arguments that
do not readily suggest any natural analog in the case of time-focused selection.

There are also several other size-focused selection problems where it may be informative to investigate the corresponding
time-focused problem. The most natural of these is probably
the unimodal subsequence selection problems where one considers the random variable
\begin{align*}
U^o_n(\pi_n)=\max  \{ k:\; & X_{\tau_1} <  X_{\tau_2} < \cdots < X_{\tau_t} > X_{\tau_{t+1}} > \cdots >  X_{\tau_k}, \text{ where}\\
                       & 1\leq \tau_1 < \tau_2 < \cdots < \tau_k \leq n \},
\end{align*}
and where, as usual, each $\tau_k$ is a stopping time. \citet{ArlottoSteele:2011} found that
$$\max_{\pi_n} \E[ U^o_n(\pi_n) ]\sim 2 \sqrt{n},\quad \quad \text{as } n \rightarrow \infty,$$
and an analogous time-focused selection problem is easy to pose. Unfortunately, it does not seem easy to frame
a useful analog of Lemma \ref{lem:ind}, so the time-focused analysis of the unimodal selection problem remails open.

For another noteworthy possibility one can consider the time-focused problem in a multidimensional setting.
Here one would take the random variables
$X_i$, $i=1,2, \ldots$, to have the uniform distribution on $[0,1]^d$, and the decision maker's
task is then to select as quickly as possible
a subsequence that is monotone increasing in each coordinate.
The dual, size-focused, version of this problem was studied by \citet{BarGne:AAP2000}
who characterized the asymptotic behavior of the optimal mean.

Finally, a further possibility is the analysis of a time-focused knapsack problem.
Here one considers a sequence $X_1, X_2, \ldots$ of independent item sizes
with common continuous distribution $F$,
and the decision maker's task is to select as quickly as possible a set of $n$ items for which the total size
does not exceed a given knapsack capacity $c$.
The analytical goal is then to estimate the size of the objective function
$$
\phi_{F,c}(n) = \min_{\tau_1, \ldots, \tau_n} \E\Big[ \tau_n : \sum_{k=1}^n X_{\tau_k} \leq c \Big],
$$
where, just as before, the decision variables $\tau_1< \tau_2<  \ldots, <\tau_n$ are stopping times.
The optimal mean $\phi_{F,c}(n)$ now depends on the model distribution $F$ and on the capacity parameter $c$.
Never the less, if $F$ represents the uniform distribution on $[0,1]$ and if one takes $c=1$, then the time-focused knapsack problem
is equivalent to the time-focused increasing subsequence problem; specifically, one has
$$
\phi_{{\rm Unif},1}(n) = \beta(n), \quad \quad \text{for all } n \geq 1.
$$
The dual, size-focused, version of this problem was first studied by
\citet{CofFlaWeb:AAP1987}, \citet{RheTal:JAP1991}, and \citet{BruRob:AAP1991}
who proved asymptotic estimates for the mean.
While the structure of the optimal policy is now well understood \citep[see, e.g.][]{PapRajKle:MS1996},
only little is known about the limiting distribution of the optimal number of size-focused knapsack selections.
From \citet{ArlottoGansSteele:OR2014} one has a suggestive upper bound on the variance,
but at this point one does not know for sure that this bound gives the principle term of the variance for large $n$.

\section*{Acknowledgements}

Elchanan Mossel acknowledges the support of NSF grants DMS 1106999 and CCF 1320105, ONR grant number N00014-14-1-0823
and grant 328025 from the Simons Foundation.

The authors also thank the referees for their useful suggestions, including a clarification of the proof of
Lemma \ref{lm:betaRecursion} and the alternative optimization proof of \eqref{eq:BetaNlowerbnd} given in Section \ref{sec:mean}.

\end{document}